\documentclass[oneside, a4paper, reqno]{amsart}

\calclayout
\makeatletter 
\def\ifdraft{\ifdim\overfullrule>\z@
        \expandafter\@firstoftwo
    \else
        \expandafter\@secondoftwo
    \fi}
\makeatother 

\def\clap#1{\hbox to 0pt{\hss#1\hss}}

\usepackage[utf8]{inputenc}
\usepackage[T1]{fontenc}
\usepackage{amssymb} 
\usepackage{xcolor}
\usepackage{tikz}
\usepackage{tikz-cd}
\usepackage{mathtools} 
\usepackage{calrsfs}   

\usepackage[final]{hyperref} 
\hypersetup{
    colorlinks  = true,
    linkcolor   = blue,
    anchorcolor = blue,
    citecolor   = blue,
    filecolor   = blue,
    urlcolor    = blue,
}
\usepackage{cleveref}
\usepackage{graphicx}
\graphicspath{ {./figures/} }
\usepackage{relsize} 
\usepackage{adjustbox}

\usepackage{stackengine}
\newcommand\xrightarrowtail[2][]{\ensurestackMath{\mathrel{%
  \stackengine{1pt}{%
    \stackengine{0pt}{\rightarrowtail}{\scriptstyle#2}{O}{c}{F}{F}{S}%
  }{\scriptstyle#1}{U}{c}{F}{F}{S}%
}}}
\newcommand\xtwoheadrightarrow[2][]{
    \ensurestackMath{\mathrel{%
      \stackengine{1pt}{%
        \stackengine{0pt}{\twoheadrightarrow}{\scriptstyle#2}{O}{c}{F}{F}{S}%
      }{\scriptstyle#1}{U}{c}{F}{F}{S}%
    }}
}

\usepackage{tabstackengine}
\setstackgap{L}{.75\baselineskip}
\setstackTAB{,}
\let\qqrep\tabbedCenterstack
\newcommand\qrepbox[2]{{\parbox[c][#1em][c]{#1em}{\centering\qqrep{#2}}}}

\newcommand\qrep[1]{\qrepbox{3}{#1}}
\newcommand\qrepp[1]{\hspace{7pt}\mathclap{#1}\hspace{7pt}}

\newcommand*{\comp}{\hspace{-0.1em}\mathrel{\raisebox{.05em}{$\mathsmaller{\circ}$}}\hspace{-0.1em}}

\ifdraft{
    \usepackage{silence} 
    \usepackage[inline]{showlabels} 
    
}

\DeclareUnicodeCharacter{2211}{\Sigma}
\ifdraft{
    \DeclareUnicodeCharacter{202F}{\,}
}

\usetikzlibrary{shapes.geometric}
\tikzstyle{polyCol} = [fill, cap=round, join=round, line width=30pt]
\definecolor{myred}{rgb}{1,0.2,0.3}
\tikzset{ PO/.style = {dashed,very near start, commutative diagrams/phantom, "\lrcorner", /tikz/commutative diagrams/.cd,dr}} 
\tikzset{ labl/.style={anchor=north, rotate=90, inner sep=1mm} }
\tikzset{
  redCirc/.style={
    label={[inner sep=0,minimum size=10mm, circle, fill=myred, opacity=0.2]center:{}}
    }
}

\newtheorem{theorem}{Theorem}[section]
\newtheorem{proposition}[theorem]{Proposition}
\newtheorem{lemma}[theorem]{Lemma}
\newtheorem{corollary}[theorem]{Corollary}

\theoremstyle{definition}
\newtheorem*{notation}{Notation}

\newtheorem{definition}[theorem]{Definition}
\newtheorem{example}[theorem]{Example}
\newtheorem{remark}[theorem]{Remark}
\newtheorem{setup}[theorem]{Setup}

\ifdraft{
    
}{
    
}

\let\mod\relax

\DeclareMathOperator{\add}{add}

\DeclareMathOperator{\Ext}{Ext}

\DeclareMathOperator{\Hom}{Hom}
\DeclareMathOperator{\id}{id}

\DeclareMathOperator{\im}{im}
\DeclareMathOperator{\Iso}{Iso}

\DeclareMathOperator{\mod}{mod}

\DeclareMathOperator{\Mor}{Mor}

\DeclareMathOperator{\op}{{op}}

\DeclareMathOperator\Ab    {\mathsf{Ab}}     

\DeclareMathAlphabet{\pazocal}{OMS}{zplm}{m}{n}
\newcommand{\cA}{\mathcal{A}}

\newcommand{\cS}{\mathcal{S}}
\newcommand{\cB}{\mathcal{B}}
\newcommand{\cC}{\mathcal{C}}
\newcommand{\cE}{\mathcal{E}}
\newcommand{\cF}{\mathcal{F}}

\newcommand{\cT}{\mathcal{T}}

\newcommand{\pA}{\pazocal{A}}
\newcommand{\pC}{\pazocal{C}}

\newcommand{\pH}{\pazocal{H}}

\newcommand{\fs}{\mathfrak{s}}

\newcommand{\NN}{\mathbb{N}}
\newcommand{\EE}{\mathbb{E}}
\renewcommand{\SS}{\mathbb{S}}

\newcommand{\xto}{\xrightarrow}

\newcommand{\inv}{{^{-1}}}

\DeclarePairedDelimiterX\set[1]{\lbrace}{\rbrace}{ #1 }
\DeclarePairedDelimiterX\Set[1]{\lbrace}{\rbrace}{ #1 }

\newcommand{\xrightarrowInline}[2][0pt]{\xrightarrow{\raisebox{#1}{$\scriptstyle\smash{#2}$}}}

\newcommand\xmapsfrom[1]{\mathrel{\reflectbox{$\xmapsto{#1}$}}}

\newlength\tindent
\renewcommand{\indent}{\hspace*{\tindent}}

\ifdraft{
    \WarningFilter{latex}{Marginpar on page} 
    
}{
}

\newcommand{\dFourTorsScope}[1]{
    \begin{scope}[shift={#1}]
        \dFourTors
}
\newcommand{\dFourTors}[9]{
    \draw [gray, fill=white, fill=#1] (0,3) circle (14pt);
    \draw [gray, fill=white, fill=#2] (1,2) circle (14pt);
    \draw [gray, fill=white, fill=#3] (2,1) circle (14pt);
    \draw [gray, fill=white, fill=#4] (2,3) circle (14pt);
    \draw [gray, fill=white, fill=#5] (3,2) circle (14pt);
    \draw [gray, fill=white, fill=#6] (3,4) circle (14pt);
    \draw [gray, fill=white, fill=#7] (4,3) circle (14pt);
    \draw [gray, fill=white, fill=#8] (5,2) circle (14pt);
    \draw [gray, fill=white, fill=#9] (5,4) circle (14pt);
    \end{scope}
}

\makeatletter
\@namedef{subjclassname@2020}{%
  \textup{2020} Mathematics Subject Classification}
\makeatother


\title{Intermediate categories for proper abelian subcategories} 
\author{Anders S. Kortegaard}
\address{Department of Mathematics, Aarhus University, Ny Munkegade 118, 8000 Aarhus C, Denmark}
\email{kortegaard@math.au.dk}

\begin{document}
\maketitle
\begin{abstract}
    Let $\cA$ be an extension closed proper abelian subcategory of a triangulated category $\cT$, with no negative 1 and 2 extensions.
    From this, two functors from $\Sigma\cA\ast\cA$ to $\cA$ can be constructed giving a snake lemma 
    mirroring that of homology without needing a t-structure.
    
    \indent We generalize the concept of intermediate categories, which originates from a paper by Enomoto and Saito, to the setting of proper abelian subcategories and show that under certain assumptions this collection is in bijection with torsion-free classes in $\cA$.
\end{abstract}

\section{Introduction}
\renewcommand{\thetheorem}{\Alph{theorem}}
\setcounter{theorem}{0}

In \cite{enomoto2022grothendieck}, Enomoto and Saito introduce and study Grothendieck monoids of extriangulated categories as a generalization of the Grothendieck group.
They study what happens to the Grothendieck monoid when an extriangulated category is localized, and furthermore, they ask when this becomes the localization of the Grothendieck monoid of that original extriangulated category. 
That is, given an extriangulated category $\pC$ and a class of morphisms $S\in \Mor(\pC)$, when does the Grothendieck monoid $M(\pC_S)$ of the localization $\pC_S$ become a localization of the Grothendieck monoid $M(\pC)$ of $\pC$?
As an example of when this happens, they study \emph{intermediate categories} of a derived category $D^b(\pA)$.
A category $\pC\subseteq D^b(\pA)$ is called intermediate in $D^b(\pA)$ if it is closed under extensions and direct summands, and $\pA\subseteq \pC\subseteq \Sigma\pA\ast\pA$.

\indent In this paper we generalize the notion of intermediate categories to the setting of an arbitrary triangulated category $\cT$
containing a proper abelian subcategory $\cA$.
Proper abelian subcategories were introduced by Jørgensen in \cite{jorgensen2022abelian} as full additive subcategories $\cA$ of a triangulated category $\cT$, with conflations exactly those coming from $\cT$ as short triangles, giving $\cA$ the structure of an abelian category.
In this setting an $\cA$-\emph{intermediate category} $\pC\subseteq \cT$ is closed under extensions and direct summands, and satisfies $\cA\subseteq \pC\subseteq \Sigma\cA\ast\cA$.
Notice that this is a generalization of the former notion since $\pA$ sits inside $D^b(\pA)$ as a proper abelian subcategory.
With this new definition we can consider, as an example, hearts of t-structures which  also sit inside a triangulated category as proper abelian subcategories.
When working with a t-structure in a triangulated category $\cT$ one will always have a homology functor from $\cT$ to the heart $\pH$ of that t-structure.
This homology functor can to some extent describe objects in $\cT$ in terms of $\pH$, using long exact sequences.

\indent It is not always the case that proper abelian subcategories are hearts of t-structures;
for an example of this look no further than \Cref{ex:neg_cluster_cat}.
Therefore we loose the feature of having a homology functor.
We can make a substitute of this, using a version of the snake lemma.

\begin{lemma}[{= \Cref{prop:exact_seq}}]
    Let $\cA$ be a proper abelian subcategory in a triangulated category $\cT$.
    Assume $\cT$ is Krull-Schmidt and that $\cT(\cA,\Sigma^{-i}\cA) = 0$ for $i=1,2$.
    Let $c\in \Sigma \cA\ast\cA$, then there exists a unique minimal right $\Sigma \cA$-approximation $\Sigma a_1\rightarrow c$, and
    a unique minimal left $\cA$-approximation $c\to a_0$. 
    Defining $G,F: \Sigma \cA\ast \cA\rightarrow \cA$ by $F(c) = a_1$ and $G(c) = a_0$ induces two functors with the following property:
    Let
    \begin{equation}
        \begin{tikzcd} c\rar["f"]& c'\rar["g"]& c''\rar[]& \Sigma c \end{tikzcd}
    \end{equation}
    be a triangle in $\cT$ with $c,c',c''\in\Sigma \cA \ast\cA$,
    then there exists a morphism $\delta:F(c'')\rightarrow G(c)$, such that

    \begin{equation*}
    \begin{tikzcd}
        0      \rar&
        F(c)   \rar["F(f)"]&
        F(c')  \rar["F(g)"]&
        F(c'') \rar["\delta"]&
        G(c)   \rar["G(f)"]&
        G(c')  \rar["G(g)"]&
        G(c'')\rar&
        0
    \end{tikzcd}
    \end{equation*}
    
    is an exact sequence in $\cA$.
\end{lemma}

Using this snake lemma as a substitute for homology, we can similarly to Enomoto and Saito classify all the $\cA$-intermediate subcategories of $\cT$ as follows.

\begin{theorem}[{= \Cref{thm:intermediate_torsionfree} \& \Cref{cor:intermediate-bijection}}]
    Let $\cA$ be an extension closed proper abelian subcategory in a triangulated category $\cT$.
    Assume $\cT$ is Krull-Schmidt and that $\cT(\cA,\Sigma^{-i}\cA) = 0$ for $i=1,2$.

    \begin{enumerate}
        \item If $\cF\subseteq\cA$ is a torsion-free class then $\Sigma\cF\ast\cA$ is an $\cA$-intermediate category.
            Furthermore, $F(\Sigma\cF\ast\cA) = \cF$.
        \item Let $\pC$ be an $\cA$-intermediate category which also satisfies that $\pC\subseteq \cA\ast\Sigma \cA$. Then $F(\pC)$ is a torsion-free class.
                Furthermore $\pC=\Sigma F(\pC)\ast\cA$.
    \end{enumerate}
    If $\cA\ast\Sigma \cA = \Sigma \cA \ast \cA$,
    this gives a bijection between the collection of $\cA$-intermediate categories and that of torsion-free classes of $\cA$.
\end{theorem}

In the last section we give an example of applying this theory in a negative cluster category. 
This is a triangulated category in which proper abelian subcategories are not hearts of t-structures.
We give examples of when the assumptions made in the article are satisfied, together with applying the results.

\subsection*{Conventions}
Let $\cA$ be a proper abelian subcategory of a triangulated category $\cT$.
We will denote monomorphisms (resp. epimorphisms) in $\cA$ by $\rightarrowtail$  (resp. $\twoheadrightarrow$).
Conversely, if one of these arrows is used, it will be with respect to a proper abelian subcategory.

\indent All subcategories will by default be assumed to be full.

\section{Preliminaries}
\renewcommand{\thetheorem}{\arabic{section}.\arabic{theorem}}
\setcounter{theorem}{0}

\begin{definition}
    Let $\cT$ be a triangulated category and let $\cA,\cB\subseteq\cT$ be subcategories.
    Then define the subcategory
    \begin{equation*}
        \cA\ast\cB = \set{x \mid \text{there exists a triangle }a\to x\to b \to \Sigma a \text{ in } \cT, \text{ with }a\in\cA, b\in\cB}.
    \end{equation*}

\end{definition}

\begin{definition}
    Let $\cA$ be an additive category.
    A subcategory $\cB\subseteq \cA$ of $\cA$ is called an \emph{additive subcategory of $\cA$} if it is closed under isomorphisms, direct sums and direct summands.
\end{definition}

\subsection{Proper abelian subcategories}
Proper abelian subcategories of a triangulated category were introduced by Jørgensen in \cite{jorgensen2022abelian},
which also serves as a good introduction to the subject.
Here we will state the properties that are needed for this article.

\indent For the rest of this section let $\cT$ be a triangulated category.

\begin{definition}
    Let $\cA\subseteq\cT$ be an additive subcategory. 
    $\cA$ is a \emph{proper abelian subcategory} if it is abelian in such a way that
    $a\xrightarrowtail{\smash{\alpha}} a'\xtwoheadrightarrow{\smash{\alpha\raisebox{-1.3pt}{$\scriptstyle'$}}} a''$ is a short exact sequence in $\cA$ if and only if there is a triangle $a\xrightarrow{\smash{\raisebox{-.5pt}{$\scriptstyle\alpha$}}} a'\xrightarrow{\smash{\raisebox{-.5pt}{$\scriptstyle\alpha\raisebox{-1.3pt}{$\scriptstyle'$}$}}} a''\to \Sigma a$ in $\cT$.
\end{definition}

\begin{definition}[{\cite[def. 0.2]{jorgensen2021proper}}]
    Let $\cA\subseteq\cT$ be a proper abelian subcategory, $n\in\NN$.
    We say that $\cA$ \emph{satisfies $E_n$}
    if $\cT(\cA,\Sigma^{-i}\cA) = 0$  for $0<i\leq n$.
\end{definition}

\begin{definition}
    Let $\cA$ be an abelian category. An additive subcategory $\cF\subseteq \cA$ is called a \emph{torsion-free} class if it is closed under extensions and subobjects.
\end{definition}

\begin{lemma}
    \label{lem:proper_ab_res}
    Let $\cA\subseteq\cT$ be an extension closed proper abelian subcategory, $\cF\subseteq \cA$ a torsion-free class then the following statements hold.
    \begin{enumerate}
        \item\label{item:par1} $\cA\ast\Sigma\cF\subseteq\Sigma\cF\ast\cA$,
        \item\label{item:par2} $\Sigma\cF\ast\cA$ is extension closed,
        \item\label{item:par3} If $\cT$ is Krull-Schmidt and $\cA$ satisfies $E_1$, then $\Sigma\cF\ast\cA$ is an additive subcategory of $\cT$.
    \end{enumerate}
\end{lemma}
\begin{proof}
    (\ref{item:par1}) Since $\cF$ is closed under subobjects, this follows from a similar argument as \cite[lem. 5.2]{jorgensen2022abelian}.

    \medbreak
    (\ref{item:par2}) Follows directly from the following calculation:
    \begin{equation*}
        (\Sigma \cF \ast \cA)\ast(\Sigma \cF \ast \cA)
        \subseteq (\Sigma \cF \ast \Sigma \cF)\ast( \cA \ast \cA)
        \subseteq \Sigma (\cF \ast \cF)\ast(\cA \ast \cA)
        \subseteq \Sigma \cF\ast \cA.
    \end{equation*}

    \medbreak
    (\ref{item:par3}) Since $\cF$ and $\cA$ are both additive subcategories of $\cT$, it follows directly from \cite[2.1(i)]{iyama2008mutation} that $\Sigma \cF\ast\cA$ is an additive subcategory of $\cT$.
\end{proof}

\subsection{Extriangulated categories}
The concept of extriangulated categories is a simultaneous generalization of exact categories and triangulated categories.
These were first introduced by Nakaoka and Palu in \cite{nakaoka2019extriangulated}. 
For the concrete definition see \cite[def.~2.12]{nakaoka2019extriangulated}.
An \emph{extriangulated category} is a triple $\cC = (\cC, \EE, \fs)$, where $\cC$ is an additive category,
$\EE: \cC^{\op}\times\cC\to\Ab$ is a bifunctor to the category of abelian groups, and $\fs$ is a ``realization'' of $\EE$.

\begin{example}
    \begin{enumerate}
        \item Let $\cE$ be an exact category, then $\cE$ can be described as an extriangulated category $(\cE, \Ext_\cE^1(-,-),\fs)$
            where the realization of $\delta\in\Ext_\cE^1(Z,X)$ is its corresponding conflation,
            i.e. $\fs(\delta) = (X\rightarrowtail Y\twoheadrightarrow Z)$.
        \item Let $\cT$ be a triangulated category, then $\cT$ can be described as an extriangulated category $(\cT, \Hom_\cT(-,\Sigma -),\fs)$
            where the realization of $\delta\in\Hom_\cE(Z,\Sigma X)$ is its corresponding short triangle,
            i.e. $\fs(\delta) = (X\to Y\to Z)$.
        \item Let $\cA$ be an extension closed proper abelian subcategory of a triangulated category $\cT$.
            The canonical extriangulated structure of $\cA$ is the restriction of the extriangulated structure on $\cT$ to $\cA$.
    \end{enumerate}
\end{example}

There are also extriangulated categories that are neither exact nor triangulated.
The following lemma can give an example of such a category.

\begin{lemma}
    Let $\cA$ be an extension closed proper abelian subcategory of a triangulated category $\cT$.
    Then  $\Sigma \cA\ast\cA = (\Sigma \cA\ast\cA, \Hom_\cT(-,\Sigma -), \fs)$
    is an extriangulated category obtained by restriction of the extriangulated structure of $\cT$.
\end{lemma}
\begin{proof}
    By \Cref{lem:proper_ab_res}(\ref{item:par2}), $\Sigma \cA\ast\cA$ is extension-closed. 
    Therefore the claim follows directly from \cite[rmk.~2.18]{nakaoka2019extriangulated}.
\end{proof}

\subsection{Monoids}
For an introduction to monoids and localization of monoids we refer to \cite[app.~A]{enomoto2022grothendieck}.
Here we give a very short introduction including the results we need.

\begin{definition}
    A \emph{monoid} is a pair $(M, \cdot)$, where $M$ is a set, and $\cdot:M \times M \to M$ is a binary operation satisfying the following:
    \begin{enumerate}
        \item The operation ``$\cdot$'' is associative,
        \item There is an identity element in $M$ w.r.t. ``$\cdot$''.
    \end{enumerate}
    In other words, a monoid can be viewed as a group without inverses.
\end{definition}

\begin{definition}
    Given a monoid $M$ and a subset $S\subseteq M$, a \emph{localization} of $M$ w.r.t. $S$
    is a pair $(M_S, q)$ with a monoid $M_S$ and a morphism $q: M\to M_S$ satisfying the following universal property:
    For each monoid $N$ and morphism $\phi:M\to N$ where $\phi(s)$ is invertible for all $s\in S$, there exists
    a unique morphism $\psi:M_S\rightarrow N$ such that $\phi = \psi \comp q$.
    That is, there exists a commutative diagram

    \begin{equation*}
    \begin{tikzcd}
        M\dar["q"']\rar["\phi"] & N\\
        M_S\ar[ur, "\!\exists!\,\psi"', dashed]
    \end{tikzcd}
    \end{equation*}
\end{definition}

\begin{lemma}[{\cite[under def. A.11]{enomoto2022grothendieck}}]
    Given a monoid $M$ and a subset $S\subseteq M$, the localization $M_S$ exists.
\end{lemma}

\subsection{The Grothendieck monoid}

\begin{definition}[{\cite[def. 2.5]{enomoto2022grothendieck}}]
    Let $\pC$ be a skeletally small extriangulated category. 
    Then a \emph{Grothendieck monoid} of $\pC$ is a pair $M(\pC) = (M(\pC), \pi)$,
    with a monoid $M(\pC)$, and a map $\pi:\Iso(\pC)\to M(\pC)$, such that

    \begin{enumerate}
        \item $\pi$ \emph{respects conflations}, i.e. $\pi([0]) = 0$ and for every conflation $X\to Y\to Z\dashrightarrow$ in $\pC$ one has that
            $\pi([X])+\pi([Z]) = \pi([Y])$.
        \item Given a monoid $N$ and map $\mu:\Iso(\pC)\to N$ which respects conflations,
            there exists a unique morphism $\overline{\mu}:M(\pC)\to N$ such that $\overline{\mu}\comp\pi=\mu$.
            
            \begin{equation*}
            \begin{tikzcd}
                \Iso(\pC)\dar["\pi"']\rar["\mu"] &  N\\
                M(\pC).\ar[ur, "\!\exists!\,\overline{\mu}"', dashed]
            \end{tikzcd}
            \end{equation*}
    \end{enumerate}
\end{definition}

\begin{proposition}[{\cite[prop.~2.7]{enomoto2022grothendieck}}]
    Let $\pC$ be an skeletally small extriangulated category, then the Grothendieck monoid exists.
\end{proposition}

\subsection{Intermediate categories} 

\begin{definition}
    Let $\cA$ be an extension closed proper abelian subcategory of a triangulated category $\cT$.
    A subcategory $\pC\subseteq \cT$ is called an \emph{$\cA$-intermediate category} if
    \begin{enumerate}
        \item $\cA\subseteq \pC \subseteq \Sigma \cA\ast\cA$,
        \item $\pC$ is extension-closed,
        \item $\pC$ closed under direct summands.
    \end{enumerate}
\end{definition}

\begin{lemma}
    Let $\cA$ be an extension closed proper abelian subcategory of a triangulated category $\cT$,
    then an $\cA$-intermediate category $\pC$ has an extriangulated structure inherited from $\cT$ by restriction.
\end{lemma}
\begin{proof}
    By definition $\pC$ will be extension-closed, hence it follows directly from \cite[rmk.~2.18]{nakaoka2019extriangulated}
    that $\pC$ has the claimed extriangulated structure.
\end{proof}

\section{The snake lemma}

The following setup will be assumed throughout the rest of this section.
\begin{setup}
    \label{setup}
    Let $\cT$ be a Krull-Schmidt triangulated category, and $\cA$ an extension closed proper abelian subcategory satisfying $E_2$.
\end{setup}

\begin{lemma} We have the following.
    \label{lem:GF_definition}
    \begin{enumerate}
        \item \label{item:GF_1} Up to isomorphism each $x\in\Sigma \cA\ast\cA$ permits a unique short triangle $\Sigma a_1^x \xrightarrow{\phi_x} x \xrightarrow{\psi_x} a_0^x$, with $\phi_x$ a minimal right $\Sigma \cA$-approximation, and $\psi_x$ a minimal left $\cA$-approximation. Furthermore, $\phi_x, \psi_x$ are natural in $x$.
        \item \label{item:GF_2} The assignments $x\mapsto a_1^x$ and $x\mapsto a_0^x$ have canonical augmentations to functors $F$ and $G$.
    \end{enumerate}
\end{lemma}
\begin{proof}
    Given $x\in\Sigma\cA\ast\cA$ there is a short triangle $\Sigma a_1^x \xrightarrow{\phi_x} x \xrightarrow{\psi_x} a_0^x$, with $a_0^x, a_1^x\in\cA$.
        Since $\cA$ satisfies $E_1$ we get that $\phi_x$ is a right $\Sigma \cA$-approximation, and
        $\psi_x$ is a left $\cA$-approximation.
        Since $\cT$ is Krull-Schmidt there are, up to isomorphism, unique $a_0^x,a_1^x,\psi_x,\phi_x$ such that $\phi_x,\psi_x$ are minimal.
        Before showing that $\phi_x$ and $\psi_x$ are natural in $x$, we will show that the assignments $F(x) = a_1^x$ and $G(x) = a_0^x$ induce functors.
    We will show that $F$ is a functor. Showing that $G$ is a functor follows by a similar argument.
    
    \indent It is enough to check that given $x,y\in\Sigma\cA\ast\cA$, and $f\in\Hom(x,y)$,  there
    exists a unique morphism $\alpha:\Sigma F(x)\rightarrow \Sigma F(y)$ such that the following diagram commutes,

    \begin{equation}
    \label{eq:def_Ff}
    \begin{tikzcd}
        \Sigma F(x) \rar["\alpha"] \dar["\phi_x"]& 
        \Sigma F(y) \dar["\phi_y"]\\
        x \rar["f"]& 
        y,
    \end{tikzcd}
    \end{equation}

    where $\phi_x,\phi_y$ are the minimal right $\Sigma\cA$-approximations from above.
    Existence of $\alpha$ follows directly from $\phi_y$ being a right $\Sigma \cA$-approximation.
    To prove uniqueness, assume there is another morphism $\beta:\Sigma F(x)\rightarrow \Sigma F(y)$ satisfying that $\phi_y \comp \beta = f\comp\phi_x$.
    Then $\phi_y \comp (\alpha - \beta) = 0$,
    thus $\alpha-\beta$ factors though $\Sigma\inv G(y)$:

    \begin{equation*}
    \begin{tikzcd}
        & \Sigma F(x)\rar["\phi_x"]\dar["\alpha - \beta"]\ar[dl, dashed] & x\dar["0"]\\
        \Sigma\inv G(y)\rar& \Sigma F(y)\rar["\phi_y"] & y.
    \end{tikzcd}
    \end{equation*}
    
    The assumption that $\cA$ satisfies $E_2$ implies that $\alpha - \beta = 0$, thus $\alpha = \beta$.
    This shows that $F$ is a functor, and that $\phi_x$ is natural in $x$ follows directly from this (see \cref{eq:def_Ff}).
\end{proof}

\begin{notation}
For the rest of the paper we let $F,G$ denote the functors from \Cref{lem:GF_definition}, and given $x\in\Sigma\cA\ast\cA$ we let $\phi_x,\psi_x$ denote the morphisms from \Cref{lem:GF_definition}.
\end{notation}

\begin{lemma}
    \label{prop:exact_seq}
    Let 
    \begin{equation}
        \begin{tikzcd} c\rar["f"]& c'\rar["g"]& c''\rar& \Sigma c \end{tikzcd}
    \end{equation}
    be a triangle with $c,c',c''\in\Sigma \cA \ast\cA$
    then there exists a morphism $\delta:F(c'')\rightarrow G(c)$, such that

    \begin{equation*}
    \begin{tikzcd}
        0      \rar&
        F(c)   \rar["F(f)"]&
        F(c')  \rar["F(g)"]&
        F(c'') \rar["\delta"]&
        G(c)   \rar["G(f)"]&
        G(c')  \rar["G(g)"]&
        G(c'')\rar&
        0
    \end{tikzcd}
    \end{equation*}
    
    is an exact sequence in $\cA$.
\end{lemma}
\begin{proof}
    Given a triangle

    \begin{equation*}
    \begin{tikzcd}
        c\rar["f"]&
        c'\rar["g"]&
        c''\rar["\Sigma h"]&
        \Sigma c
    \end{tikzcd}
    \end{equation*}
    
    with $c,c',c''\in\Sigma \cA \ast\cA$, we can get the following diagram.

    \begin{equation*}
    \begin{tikzcd}
                    F(c)      \rar["F(f)"] \dar["\Sigma^{-1}\phi_c"]&
                    F(c')     \rar["F(g)"] \dar["\Sigma^{-1}\phi_{c'}"]&
                    F(c'')                 \dar["\color{black}\Sigma^{-1}\phi_{c''}", blue]&
        \Sigma      F(c)      \rar["\Sigma F(f)"] \dar["\phi_{c}"]&
        \Sigma      F(c')     \rar["\Sigma F(g)"] \dar["\phi_{c'}"]&
        \Sigma      F(c'')                 \dar["\phi_{c''}"]&\\
        \Sigma^{-1} c        \dar["\Sigma^{-1}\psi_{c}"]   \rar["\Sigma^{-1}f"]&
        \Sigma^{-1} c'       \dar["\Sigma^{-1}\psi_{c'}"]  \rar["\Sigma^{-1}g"]&
        \Sigma^{-1} c''      \dar["\Sigma^{-1}\psi_{c''}"] \rar["\color{black}h", blue]&
                    c        \dar["\color{black}\psi_{c}", blue]              \rar["f"]&
                    c'       \dar["\psi_{c'}"]             \rar["g"]&
                    c''      \dar["\psi_{c''}"]\\
        \Sigma^{-1} G(c)     \rar["\Sigma^{-1} G(f)"]&
        \Sigma^{-1} G(c')    \rar["\Sigma^{-1} G(g)"]&
        \Sigma^{-1} G(c'')&
                    G(c)     \rar["G(f)"]&
                    G(c')    \rar["G(g)"]&
                    G(c'').
    \end{tikzcd}
    \end{equation*}

    Composing the blue arrows in the diagram above defines $\delta\coloneqq \psi_{c}\comp h\comp \Sigma^{-1}\phi_{c''}\colon F(c'')\rightarrow G(c)$.
    \medbreak

    \underline{Exact at $F(c)$.} Since $\cA$ is an abelian category there exists a conflation

    \begin{equation*}
    \begin{tikzcd}
        \ker F(g) \rar["\beta", rightarrowtail]&
        F(c') \rar["\alpha'", twoheadrightarrow]&
        \im F(g).
    \end{tikzcd}
    \end{equation*}

    By definition this conflation comes from a triangle in $\cT$.
    Consider the following commutative diagram of solid arrows.

    \begin{equation*}
    \begin{tikzcd}
        F(c)      \ar[dd,"\Sigma^{-1}\phi_{c}"']  \ar[rr,"F(f)"] \ar[dr, "\alpha", bend left=20, dashed]&&
        F(c')     \ar[dd,"\Sigma^{-1}\phi_{c'}"]  \ar[rr,"F(g)"] \ar[dr, "\alpha'", twoheadrightarrow] &&
        F(c'')    \ar[dd,"\Sigma^{-1}\phi_{c''}"]              \\ 
        & 
        \ker F(g) \ar[dl,"\varepsilon", dashed]\ar[ur,"\beta", rightarrowtail]\ar[ul, "\eta", bend left=20, dashed]&& 
        \im F(g) \ar[dr]\ar[ur,"\beta'", rightarrowtail]&\\
        \Sigma^{-1} c   \ar[rr, "\Sigma^{-1}f"]&&
        \Sigma^{-1} c'  \ar[rr, "\Sigma^{-1}g"]&&
        \Sigma^{-1} c''.
    \end{tikzcd}
    \end{equation*}
    
    Using the axiom TR3 we get that $\varepsilon$ exists.
    Since $\Sigma^{-1}\phi_{c}$ is a right $\cA$-approximation, $\eta$ exists such that $\Sigma^{-1} \phi_{c}\comp\eta = \varepsilon$.
    From the fact that $F(g)\comp F(f) = 0$ it follows that $\alpha$ exists in such a way that $\beta\comp\alpha = F(f)$.

    \indent We will check that $\varepsilon\comp \alpha = \Sigma^{-1}\phi_{c}$.
    Notice that 
    \begin{equation*}
        \Sigma^{-1}f\comp \varepsilon\comp \alpha 
        = \Sigma^{-1}\phi_{c'}\comp \beta\comp \alpha
        = \Sigma^{-1}\phi_{c'}\comp F(f)
        = \Sigma^{-1}f\comp\Sigma^{-1}\phi_{c}.
    \end{equation*}

    Therefore $\Sigma^{-1}f(\varepsilon \comp \alpha - \Sigma^{-1}\phi_{c}) = 0$. 
    This means that $\varepsilon \comp \alpha - \Sigma^{-1}\phi_{c}$ factors through $\Sigma^{-2}c''$.

    \begin{equation*}
    \begin{tikzcd}
        & F(c)\dar["\varepsilon \comp \alpha - \Sigma^{-1}\phi_{c}"]\ar[dr,"0", bend left=35]\ar[dl,dashed, bend right=35] &\\
        \Sigma^{-2}c''\rar &
        \Sigma^{-1}c\rar &
        \Sigma^{-1}c'.
    \end{tikzcd}
    \end{equation*}

    Since $\cA$ satisfies $E_2$, we have that $\Hom(\cA, \Sigma^{-1}\cA\ast\Sigma^{-2}\cA)=0$, 
    thus $\varepsilon \comp \alpha - \Sigma^{-1}\phi_{c}= 0 $, giving that $\varepsilon\comp \alpha = \Sigma^{-1}\phi_{c}$.
    \medbreak

    We now check that $\alpha\comp\eta = \id$. 
    Firstly we have that 
    \begin{equation*}
        \Sigma^{-1}\phi_{c'}\comp\beta\comp\alpha\comp\eta 
        = \Sigma^{-1}\phi_{c'}\comp F(f)\comp\eta
        = \Sigma^{-1}f\comp\Sigma^{-1}\phi_{c}\comp\eta
        = \Sigma^{-1}f\comp\varepsilon
        = \Sigma^{-1}\phi_{c'}\comp\beta.
    \end{equation*}

    Hence $\Sigma^{-1}\phi_{c'}\comp(\beta -\beta\comp\alpha\comp \eta) = 0$.
    A similar argument as before shows that $\beta -\beta\comp\alpha\comp \eta = 0$
    since it would have to factor through $\Sigma^{-2}G(c')$.
    Therefore $\beta(\id - \alpha\comp \eta) = 0$, and since $\beta$ is a monomorphism in $\cA$, this implies that $\id - \alpha\comp \eta = 0$, and therefore $\id = \alpha\comp \eta$.

    \medbreak
    Lastly we see that $\eta\comp\alpha = \id$. Notice that 

    \begin{equation*}
        \Sigma^{-1}\phi_{c}\comp \eta\comp\alpha 
        = \varepsilon\comp\alpha 
        = \Sigma^{-1}\phi_{c}.
    \end{equation*}

    Thus $\Sigma^{-1}\phi_{c}\comp(\eta\comp\alpha - \id) = 0$, meaning that $\eta\comp\alpha - \id$ factors though $\Sigma^{-2}G(c)$
    and therefore $\eta\comp\alpha - \id = 0$ by the assumption that $\cA$ satisfies $E_2$.
    This concludes the proof that $\eta\comp\alpha = \id$.
    In particular, $\alpha$ is an isomorphism making $F(f)$ a monomorphism.

    \medbreak
    \underline{Exact at $F(c')$.} We just proved that $F(c)\cong\ker F(g)$.
    This gives the following diagram.
    \begin{equation*}
    \begin{tikzcd}
        F(c)    \ar[r,"F(f)", rightarrowtail] &
        F(c')   \ar[r,"F(g)"]&
        F(c'')   
    \end{tikzcd}
    \end{equation*}
    But then $\ker F(g) \cong F(c) \cong \im F(f)$, making the sequence exact at $F(c')$.

    \medbreak
    \underline{Exact at $F(c'')$.}
    Since $F(f)$ is a monomorphism with cokernel $a=\im F(g)$, we get the following triangle.

    \begin{equation*}
    \begin{tikzcd}
        F(c) \rar["F(f)"] &
        F(c') \rar["\eta"] &
        a \rar["\mu"] &
        \Sigma F(c) \rar["\Sigma F(f)"] &
        \Sigma F(c').
    \end{tikzcd}
    \end{equation*}

    Using the $3\times3$~lemma (see \cite[prop. 1.1.11]{beilinson1983analyse} or \cite[lem. 2.6]{may2001additivity}) on the commutative square

    \begin{equation*}
    \begin{tikzcd}
        F(c) \rar["F(f)"]\dar["\Sigma^{-1}\phi_{c}"] &
        F(c')\dar["\Sigma^{-1}\phi_{c'}"] \\
        \Sigma\inv c \rar["\Sigma\inv f"]&
        \Sigma\inv c'
    \end{tikzcd}
    \end{equation*}
    
    gives that the following grid is a commutative diagram where each row and column is a triangle.
    In the diagram, the monomorphism $\tau$ exists such that $F(g)=\tau\eta$ because $\eta$ is a cokernel of $F(f)$ while $F(g)\comp F(f)=0$.

    \begin{equation*}
    \begin{tikzcd}[column sep=3.5em, row sep=3em]
        &&\tilde a \ar[dd, "\omega", near start]\ar[dr,"\nu",dashed]&&&\\
        F(c)         \rar["F(f)",rightarrowtail]        \ar[dd, "\Sigma^{-1}\phi_{c}"] &
        F(c')        \ar[rr,"\eta",twoheadrightarrow, crossing over, near start]      \ar[dd, "\Sigma^{-1}\phi_{c'}"] &&
        a            \rar["\mu"]         \ar[dd, "\gamma"]  \ar[dl, "\tau", rightarrowtail] &
        \Sigma F(c)  \rar["\Sigma F(f)"] \ar[dd, "\phi_{c}"] &
        \Sigma F(c')                     \ar[dd, "\phi_{c'}"] \\[-15pt]
                                                         &&F(c'')\ar[dr,"\color{black}\Sigma\inv\phi_{c''}", blue]&&&
        \\[-15pt]
        \Sigma\inv c   \ar[d, "\Sigma\inv\psi_{c}"]   \rar["\Sigma\inv f"]&
        \Sigma\inv c'  \ar[d, "\Sigma\inv\psi_{c'}"]  \ar[rr,"\Sigma\inv g"]&&
        \Sigma\inv c'' \ar[d, "\varepsilon"]      \rar["\color{black}h", blue]&
        c              \ar[d,"\color{black}\psi_{c}", blue]              \rar["f"]&
        c'             \ar[d, "\psi_{c'}"]\\
        \Sigma^{-1} G(c)     \rar["\Sigma^{-1} G(f)"]&
        \Sigma^{-1} G(c')    \ar[rr,"\xi"]&&
         x\rar["\zeta"]&
                    G(c)     \rar["G(f)"]&
                    G(c').   
    \end{tikzcd}
    \end{equation*}

    Recall that $\delta = \psi_{c}\comp h\comp \Sigma^{-1}\phi_{c''}$ (i.e. the composition of the blue arrows in the diagram above).
    We check that the solid triangle commutes, that is $\Sigma\inv\phi_{c''}\comp\tau = \gamma$.
    Since $\tau\eta = F(g)$, we get that
    \begin{equation*}
        (\gamma - \Sigma\inv\phi_{c''}\comp \tau)\eta
        = \gamma\comp\eta - \Sigma\inv\phi_{c''}\comp F(g)
        = \Sigma\inv g\comp\Sigma\inv\phi_{c'} -  \Sigma\inv g\comp\Sigma\inv\phi_{c'} 
        = 0.
    \end{equation*}

    Therefore $\gamma - \Sigma\inv\phi_{c''}\comp \tau$ factors through $\mu$:

    \begin{equation*}
    \begin{tikzcd}[column sep=3.5em]
        F(c') \rar["\eta"]\ar[dr, "0"', bend right] &
        a \rar["\mu"]\dar["\gamma - \Sigma\inv\phi_{c''}\comp \tau"] &
        \Sigma F(c)\ar[dl, dashed, bend left, "\exists\lambda"] \\
                      &\Sigma\inv(c'')&
    \end{tikzcd}
    \end{equation*}

    But notice that $\lambda\in\Hom(\Sigma\cA, \cA\ast\Sigma\inv\cA)$, thus $\lambda=0$ since $\cA$ satisfies $E_2$.
    This implies that $\gamma - \Sigma\inv\phi_{c''}\comp \tau = 0$, giving that $\gamma = \Sigma\inv\phi_{c''}\comp \tau$.

    \medbreak
    We now show that $\tau$ satisfies the universal property of being a kernel of $\delta$.
    Let $\tilde a\in\cA$ be given together with a morphism $\omega: \tilde a\rightarrow F(c'')$ such that $\delta\comp\omega = 0$.
    We want to show that $\omega$ factors through $\tau$.

    \medbreak
    First we notice that 
    \begin{equation*}
        0 
        = \delta\comp\omega
        = \psi_{c}\comp h\comp \Sigma^{-1}\phi_{c''}\comp\omega
        = \zeta\comp (\varepsilon\comp \Sigma^{-1}\phi_{c''}\comp\omega).
    \end{equation*}

    This means that $\varepsilon\comp \Sigma^{-1}\phi_{c''}\comp\omega$ factors through $\xi$.
    But $\Hom(\tilde a, \Sigma\inv G(c'))=0$ and therefore $\varepsilon\comp \Sigma^{-1}\phi_{c''}\comp\omega = 0$.
    This gives that $\Sigma^{-1}\phi_{c''}\comp\omega$ factors through $\gamma$, say
    $\Sigma^{-1}\phi_{c''}\comp\omega = \gamma \comp \nu$ ($\nu$ is the dashed morphism in the diagram).
    Notice that

    \begin{equation*}
        \Sigma\inv\phi_{c''}(\omega - \tau\comp\nu)
        = \Sigma\inv\phi_{c''}\comp \omega -  \gamma\nu 
        = 0.
    \end{equation*}

    Thus $\omega - \tau\comp\nu$ factors through $\Sigma^{-2} G(c'')$.

    \begin{equation*}
    \begin{tikzcd}[column sep = 3.5em]
        &\tilde a\ar[d, "\omega - \tau\comp\nu"]\ar[dr, bend left = 20, "0"]\ar[dl, bend right = 20, dashed, "0"']&\\
        \Sigma^{-2} G(c'') \rar[]&F(c'')\rar["\Sigma\inv \phi_{c''}"']& \Sigma\inv c''.
    \end{tikzcd}
    \end{equation*}

    This means that $\omega - \tau\comp\nu= 0 $, thus $\omega = \tau\comp\nu$.
    The uniqueness part of the universal property follows from $\tau$ being a monomorphism in $\cA$. This concludes the proof that $\ker\delta \cong a$.

    \medbreak
    \underline{Exact at $G(c),G(c'),G(c'')$.}
    These can be done in a dual way to showing that the sequence is exact in $F(c),F(c'),F(c'')$.
\end{proof}

\begin{lemma}
    \label{lem:intermediate_epi}
    Let $c\in\cA\ast\Sigma \cA$, and thus $c\in\Sigma \cA\ast\cA$ by \Cref{lem:proper_ab_res}(\ref{item:par1}).
    Given two triangles $a_0 \xrightarrowInline[-.5pt]{f} c \rightarrow \Sigma a_1\rightarrow \Sigma a_0$
    and $\Sigma b_0 \rightarrow c \xrightarrowInline[-.5pt]{\beta} b_1 \rightarrow \Sigma^2 b_0$, with $a_i, b_i\in\cA$, the composition $\beta\comp f:a_0\rightarrow b_1$ is an epimorphism in $\cA$.
\end{lemma}

\begin{proof}
    There exists a diagram

    \begin{equation*}
    \begin{tikzcd}
        &a_0\dar["f"]&&\\
        \Sigma b_0\rar["\alpha"] &
        c\rar["\beta"]\dar["m"]&
        b_1\rar["\gamma"]\dar["h"]&
        \Sigma^2 b_0\ar[dl, "\varepsilon\ =\ 0",dashed]\\
        &\Sigma a_1\rar["\xi\ =\ 0", dashed]&d&
    \end{tikzcd}
    \end{equation*}

    where the row and column containing $c$ are triangles.
    We claim that $\beta\comp f : a_0\to b_1$ is an epimorphism.
    To show this let $d\in\cA$ and a morphism $h:b_1\to d$ be given such that $h\comp\beta\comp f = 0$.
    It is enough to show that $h = 0$.
    Since $h\comp\beta\comp f = 0$ we get that $h\comp\beta$ factors through $m$,  
    i.e. there exists a morphism $\xi:\Sigma a_1 \to d$, such that $\xi\comp m = h \comp \beta$.
    But notice that $\xi = 0$ since $\cA$ satisfies $E_1$. Thus $h \comp \beta = 0 $.
    This means that $h$ factors through $\gamma$, i.e. there exists a morphism $\varepsilon:\Sigma^2b_0\to d$
    such that $h = \varepsilon \comp \gamma$. But since $\cA$ satisfies $E_2$, we get that $\varepsilon = 0$, hence $h = 0$.
\end{proof}

\begin{proposition}
    \label{lem:fac_epi_com}
    The following are equivalent.
    \begin{enumerate}
        \item For all $a,a'\in \cA$ and $f\in\cT(a,\Sigma^2a')$ there exist $d\in\cA$, $g_1\in\cT(a,\Sigma d)$, and $g_2\in\cT(\Sigma d, \Sigma^2 a')$
            such that $f=g_2\comp g_1$.
        \item $\Sigma\cA\ast\cA=\cA\ast\Sigma\cA$.
        \item For $c\in\Sigma\cA\ast\cA$, there exist $a\in\cA$ and $f\in\cT(a, c)$,
            such that $\psi_c \comp f:a\rightarrow G(c)$ is an epimorphism.
    \end{enumerate}
\end{proposition}
\begin{proof}
    \underline{$(1)\Rightarrow (2)$}: 
    By \Cref{lem:proper_ab_res}(\ref{item:par1}), $\Sigma\cA\ast\cA\supseteq\cA\ast\Sigma\cA$.
    To check the other inclusion let $c\in \Sigma\cA\ast\cA$, which means that there is a triangle

    \begin{equation*}
    \begin{tikzcd}
        \Sigma F(c)\rar[]&
        c\rar[]&
        G(c)\rar["\gamma"]&
        \Sigma^{2}F(c).
    \end{tikzcd}
    \end{equation*}

    By (1) there exists $d\in \cA$ and morphisms $g_1: G(c)\to\Sigma d$, $g_2:\Sigma d\to \Sigma^2 F(c)$,
    such that $g_2\comp g_1 = \gamma$.
    Using the  octahedral axiom (see \cite[prop 1.4.6]{neeman2014triangulated}) we get a diagram

    \begin{equation*}
    \begin{tikzcd}
        d \rar[] \dar&
        M \rar[] \dar&
        G(c) \rar["g_1"]\dar[equal] &
        \Sigma d\dar["g_2"]\\
        \Sigma F(c)\rar[]\dar &
        c \rar[]\dar &
        G(c) \rar["\gamma"] \dar&
        \Sigma^2 F(c)\dar\\
        N\rar[equal]\dar &
        N \rar[]\dar &
        0 \rar[]\dar &
        \Sigma N\dar\\
        \Sigma d\rar[] &
        \Sigma M \rar[] &
        \Sigma G(c)\rar[] &
        \Sigma^2 d,
    \end{tikzcd}
    \end{equation*}

    where each row and column is a triangle.
    Since $\cA$ and $\Sigma\cA$ are extension closed in $\cT$ we get that $M\in \cA$ and $N\in\Sigma \cA$.
    The second column in the diagram now implies that $c\in\cA\ast\Sigma\cA$.
    
    \medbreak
    \underline{$(2)\Rightarrow (1)$}: 
    Let $f\in\cT(a,\Sigma^2 a')$, and consider the triangle

    \begin{equation*}
    \begin{tikzcd}
        \Sigma a' \rar&
        c \rar&
        a \rar["f"]&
        \Sigma^2 a'.
    \end{tikzcd}
    \end{equation*}

    Now $c\in\Sigma\cA\ast\cA=\cA\ast\Sigma\cA$, hence there exist $b_0,b_1\in \cA$ fitting into a triangle
    
    \begin{equation*}
    \begin{tikzcd}
        b_1 \rar&
        b_0 \rar&
        c \rar&
        \Sigma b_1.
    \end{tikzcd}
    \end{equation*}

    The octahedral axiom gives the following commutative diagram with rows and columns being triangles.

    \begin{equation*}
    \begin{tikzcd}
        d \rar\dar&
        b_0 \rar["h"]\dar&
        a \rar["g_1"]\dar[equal]&
        \Sigma d\dar["g_2"]\\
        \Sigma a' \rar\dar&
        c \rar\dar&
        a \rar["f"]\dar&
        \Sigma^2 a' \dar\\
        \Sigma b_1 \rar[equal]\dar&
        \Sigma b_1 \rar\dar&
        0 \rar\dar&
        \Sigma^2 b_1\dar\\
        \Sigma d \rar&
        \Sigma b_0 \rar&
        \Sigma a \rar&
        \Sigma^{2} d.
    \end{tikzcd}
    \end{equation*}

    It follows from \Cref{lem:intermediate_epi} that $h$ is an epimorphism,
    and therefore $d\in\cA$.
    This concludes the argument since now $f = g_2\comp g_1$.
    
    \medbreak
    \underline{$(2)\Rightarrow (3)$}:
    Follows directly from \Cref{lem:intermediate_epi}.

    \medbreak
    \underline{$(3)\Rightarrow (2)$}: 
    By \Cref{lem:proper_ab_res}(\ref{item:par1}), $\cA\ast\Sigma\cA \subseteq \Sigma\cA\ast\cA$,
    hence there is only left to check that $\cA\ast\Sigma\cA \supseteq \Sigma\cA\ast\cA$.
    Let $c\in\Sigma\cA\ast\cA$, and assume we have an $a\in\cA$ and a morphism $f:a\rightarrow c$ such that $\psi_c\comp f$ is an epimorphism.
    Using the octahedral axiom we can get the following commutative diagram, with rows and columns being triangles.

    \begin{equation*}
    \begin{tikzcd}
        F(c)\dar \rar &
        0\rar\dar&
        \Sigma F(c)\dar \rar[equal] &
        \Sigma F(c)\dar \\
        \Sigma\inv M \dar \rar &
        a \dar[equal] \rar["f"]&
        c \dar["\psi_c"] \rar &
        M \dar \\
        \ker(\psi_c\comp f) \rar[rightarrowtail] &
        a\rar["\psi_c\,\comp\, f", twoheadrightarrow]&
        G(c) \rar &
        \Sigma \ker(\psi_c\comp f)
    \end{tikzcd}
    \end{equation*}
    
    Since the right column is a triangle, we get that $M\in\Sigma \cA$,
    and therefore since the middle row is a triangle $c\in\cA\ast\Sigma\cA$.
\end{proof}

\section{intermediate subcategories}
\Cref{setup} will be assumed throughout this section.

\begin{lemma}
    \label{lem:Fc_torsionless}
    Let $\cF\subseteq \cA$ be a torsion-free class. 
    Let $c \in \Sigma\cF\ast\cA$, then $F(c)\in \cF$.
\end{lemma}
\begin{proof}
    Let $c \in \Sigma\cF\ast\cA$, then there is a triangle $\Sigma f\xto{\alpha} c\to a$, where $f\in\cF$ and $a\in\cA$.
    Since $\alpha$ is a right $\Sigma\cA$-approximation we get that the object $F(c)$ from the minimal right $\Sigma\cA$-approximation is a direct summand of $f$, and therefore $F(c)\in \cF$.
\end{proof}

\begin{theorem}[{cf. \cite[thm.~5.3]{enomoto2022grothendieck}}]
    \label{thm:intermediate_torsionfree}
    The following statements hold.
    \begin{enumerate}
        \item \label{itm:it1} If $\cF\subseteq\cA$ is a torsion-free class then $\Sigma\cF\ast\cA$ is an $\cA$-intermediate category.
            Furthermore, $F(\Sigma\cF\ast\cA) = \cF$.
        \item \label{itm:it2} Let $\pC$ be an $\cA$-intermediate category such that $\pC\subseteq \cA\ast\Sigma \cA$. Then $F(\pC)$ is torsion-free.
                Furthermore, we have that $\pC=\Sigma F(\pC)\ast\cA$.
    \end{enumerate}
\end{theorem}
\begin{proof}
    (\ref{itm:it1}) It follows directly from \Cref{lem:proper_ab_res} that $\Sigma\cF\ast\cA$ is an $\cA$-intermediate category.  
    Furthermore, the claim that $F(\Sigma\cF\ast\cA)=\cF$ follows from \Cref{lem:Fc_torsionless},
        together with the fact that $F(\Sigma f) = f$ for all $f\in\cF$.

    \medbreak
    (\ref{itm:it2})
    That $\pC\subseteq \Sigma F(\pC)\ast \cA$ follows directly from the fact that there for each $c\in\pC$ is a triangle

    \begin{equation*}
    \begin{tikzcd}
        \Sigma F(c)\rar& c\rar&G(c)\rar&\Sigma^2 F(c).
    \end{tikzcd}
    \end{equation*}

    To see that $\pC\supseteq \Sigma F(\pC)\ast \cA$ it is enough to check that $\Sigma F(\pC)\subseteq \pC$.
    Let $c\in\pC$. 
    We will check that $\Sigma F(c)\in\pC$. There is a triangle

    \begin{equation*}
    \begin{tikzcd}
        \Sigma F(c)\rar["\phi_c"]& 
        c\rar["\psi_c"]&
        G(c)\rar&
        \Sigma^2 F(c).
    \end{tikzcd}
    \end{equation*}

    Since $\pC\subseteq \cA\ast\Sigma\cA$, \Cref{lem:intermediate_epi} says that there exists an object $b\in \cA$ and a morphism $\alpha: b\rightarrow c$
    such that $\psi_c\comp\alpha : b\to G(c)$ is an epimorphism.
    Consider the following commutative square

    \begin{equation*}
        \begin{tikzcd}[ampersand replacement = \&]
        \Sigma F(c)\rar[equal]\dar["\begin{psmallmatrix}0\\1\end{psmallmatrix}"]\&
        \Sigma F(c)\dar["\phi_c"]\\
        b\oplus\Sigma F(c)\rar["{\begin{psmallmatrix}\alpha&\phi_c\end{psmallmatrix}}"]\&
        c.
    \end{tikzcd}
    \end{equation*}

    Using the octahedral axiom we can complete this into a commutative diagram

    \begin{equation*}
    \begin{tikzcd}[ampersand replacement = \&]
        0\rar\dar\&
        \Sigma F(c)\rar[equal]\dar["\begin{psmallmatrix}0\\1\end{psmallmatrix}"]\&
        \Sigma F(c)\dar["\phi_c"]\\
        M\rar\dar[equal]\&
        b\oplus\Sigma F(c)\rar["{\begin{psmallmatrix}\alpha&\phi_c\end{psmallmatrix}}"]
            \dar["{\begin{psmallmatrix}1&0\end{psmallmatrix}}"]\&
        c\dar["\psi_c"]\\
        M\rar\&
        b\rar["\delta"]\&
        G(c),
    \end{tikzcd}
    \end{equation*}

    where the rows and columns are short triangles. 
    Since the lower right square commutes we get that $\delta = \psi_c\comp \alpha$ is an epimorphism.
    This implies that $M\in \cA$.
    Using that the middle row in the diagram is a triangle we get that $b\oplus \Sigma F(c)\in\pC$ due to $\pC$ being extension closed.
    In particular, this means that $\Sigma F(c)\in\pC$ given that $\pC$ is closed under direct summands.
    This implies that $\pC = \Sigma F(c)\ast \cA$.

    \medbreak
    To show that $F(\pC)$ is a torsion-free class, we first need to check that it is extension closed.
    Let $c,c'\in\pC$ and assume that there is a conflation
    \begin{equation*}
    \begin{tikzcd}
        F(c)\rar[rightarrowtail]&
        d\rar[twoheadrightarrow]&
        F(c')
    \end{tikzcd}
    \end{equation*}
    
    in $\cA$.
    We just saw that $\Sigma F(c), \Sigma F(c')\in \pC$, and since $\pC$ is extension closed this implies that $\Sigma d\in\pC$. 
    Since $d\in\cA$ we have $d=F(\Sigma d)$, thus $d\in F(\pC)$.
    Lastly to show that $F(\pC)$ is closed under subobjects, let $F(c)\in F(\pC)$, and let $d'\in\cA$ be a subobject.
    This gives a triangle
    \begin{equation*}
    \begin{tikzcd}
        a\rar&
        \Sigma d'\rar&
        \Sigma F(c)\rar&
        \Sigma a
    \end{tikzcd}
    \end{equation*}

    with $a\in\cA$.
    Since $\pC$ is closed under extensions, $\Sigma d'\in\pC$.
    Given that $\Sigma d'\in\Sigma \cA$, we get that $G(\Sigma d') = 0$ and thus $\Sigma F(\Sigma d') = \Sigma d'$.
    Therefore $d' = F(\Sigma d')\in F(\pC)$.
\end{proof}

\begin{corollary}
    \label{cor:intermediate-bijection}
    If $\Sigma \cA\ast\cA=\cA\ast\Sigma \cA$
    then there is a bijection
    \begin{align*}
        \set{\pC\subseteq \cT\mid \pC \text{ is }\cA\text{-intermediate}} &\xleftrightarrow{\hspace{.5em}1:1\hspace{.5em}} \set{\cF\subseteq\cA\mid \cF \text{ torsion-free}}\\[-3pt]
        \pC&\xmapsto{\phantom{\hspace{.5em}1:1\hspace{.5em}}} F(\pC)\\[-5pt]
        \Sigma\cF\ast\cA&\xmapsfrom{\phantom{\hspace{.5em}1:1\hspace{.5em}}} \cF
    \end{align*}
\end{corollary}
\begin{proof}
    Follows directly from \Cref{thm:intermediate_torsionfree}.
\end{proof}

\begin{remark}
    \Cref{cor:intermediate-bijection} uses the assumption that $\Sigma \cA \ast \cA = \cA \ast\Sigma \cA$.
    Notice that \Cref{lem:fac_epi_com} gives some statements that are equivalent to this. 
\end{remark}

\begin{theorem}
   Assume that $\cT$ is skeletally small. Let $\cF\subseteq \cA$ be a torsion-free class, then the monoid morphism $i:M(\cA)\to M(\Sigma\cF\ast\cA)$ induced by the inclusion $\cA\to \Sigma\cF\ast\cA$  
   induces an isomorphism
   \begin{equation*}
       M(\cA)_{M_\cF}\to M(\Sigma\cF\ast\cA),
   \end{equation*}
   where $M_\cF = \set{[x]\in M(\cA) \mid x\in \cF}$.
\end{theorem}
\begin{proof}
    Using \Cref{prop:exact_seq}, this proof is exactly the same as in \cite[thm. 5.4]{enomoto2022grothendieck}.
\end{proof}

\section{Examples}
    \begin{lemma}[{\cite[lem.~3.1]{holm2013sparseness}}]
        \label{lem:neg_cy_triv_heart}
        Let $\cC$ be a $(-w)$-CY category for $w\in\NN$. Then for each t-structure $(X,Y)$ the associated heart $H=X\cap\Sigma Y = 0$.
    \end{lemma}

    Given integers $w\geq 1$, $n\geq 1$, one can define the negative cluster category as the orbit category
    \begin{equation*}
        \pC_{-w}(A_n) \coloneqq D^b(kA_n)/\Sigma^{w+1}\tau.
    \end{equation*}

    This is a triangulated category, and it is $(-w)$-Calabi--Yau, 
    i.e. there is a Serre functor given by $\SS = \Sigma^{-w}$.
    In \cite[sec. 10]{coelhosimoes2016torsion} they give a complete combinatorial model of this category.
    The indecomposable objects in $\pC_{-w}(A_n)$
    can be matched with certain diagonals in an $N$-gon, where $N=(w+1)(n+1)-2$.
    That is, if we label the vertices in the $N$-gon by $0,...,N-1$,
    we can identify an indecomposable object $X\in\pC_{-w}(A_n)$ as a pair of numbers $X = (a,b)$,
    where $a$ and $b$ correspond to the endpoints of the appropriate diagonal.
    A diagonal $(a,b)$, with $a<b$, corresponds to an indecomposable if $w+1 \mid b-a+1$. 
    We call these \emph{admissible diagonals}.
    See \Cref{fig:sms} for an example of such admissible diagonals.
    \medbreak
    
    There is also a combinatorial method to find proper abelian subcategories in $\pC_{-w}(A_n)$, using simple-minded systems. 
    A collection $\cS$ consisting of $n$ non-crossing admissible diagonals, with no two diagonals sharing an endpoint, corresponds to a 
    $w$-simple-minded system in $\pC_{-w}(A_n)$ (see \cite[def. 1.2]{jorgensen2022abelian} for a definition).
    given such a simple-minded system $\cS$, one can obtain the corresponding abelian category $\langle\cS\rangle$ by closing it under extensions. 
    See \Cref{fig:sms} for an example of such a collection.

    \begin{figure}[ht]
        \begin{tikzpicture}
            \node[regular polygon,regular polygon sides=10,minimum size=4cm,draw] (a){};
            \node[regular polygon,regular polygon sides=10,minimum size=4.5cm,draw, opacity=0] (b){};
            \foreach \x in {1,...,9}{\node[] at (b.corner \x) {\x};}
            \node[] at (b.corner 10) {0};
            \draw[] (a.corner 1) -- (a.corner 3);
            \draw[] (a.corner 10) -- (a.corner 5);
            \draw[] (a.corner 7) -- (a.corner 9);
        \end{tikzpicture}
        \caption{A simple-minded systems for $\pC_{-2}(A_3)$.}
        \label{fig:sms}
    \end{figure}
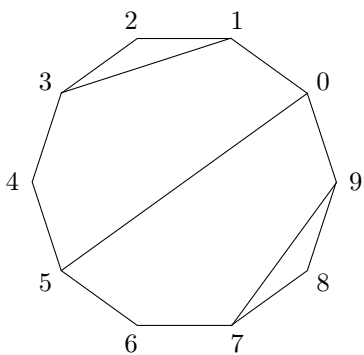

    \medbreak
    Notice that due to \Cref{lem:neg_cy_triv_heart} there is no t-structure with a non-trivial heart in any of the negative cluster categories.
    This means that the proper abelian subcategories obtained from simple-minded systems in $\pC_{-w}(A_n)$ are not hearts of t-structures.
    
    \begin{example}
        \label{ex:neg_cluster_cat}
        Consider $\pC_{-3}(A_4)$, the AR-quiver of which can be seen in \Cref{arq:c3a4}.
        The labels of the objects correspond to admissible diagonals of an 18-gon.
        This is a category where we can find proper abelian subcategories 
        that satisfy the properties needed to use \Cref{cor:intermediate-bijection}.
        Consider the collection of indecomposable objects 
        \begin{equation*}
            \cS = \set{(0,3), (4,11), (5,8), (12,15)}.
        \end{equation*}

        This is a $3$-simple-minded system. 
        Consider the proper abelian subcategory $\cA\coloneqq \langle\cS\rangle$ induced by $\cS$.
        In \Cref{arq:c3a4} we can see the indecomposable objects of $\cA$ as those marked with red discs.
        Similarly, we can see $\Sigma \cA$ in \Cref{arq:c3a4} marked with blue discs.
        It is straightforward to check that $\cA$ satisfies $E_2$ and that 
        $\Sigma\cA \ast \cA = \cA \ast\Sigma \cA$.

    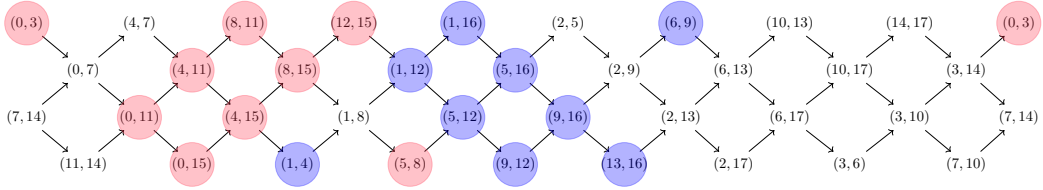
\begin{figure}[ht]
    \scalebox{.55}{
        \begin{tikzcd}[row sep=15pt, column sep = 0pt, ampersand replacement=\&,
            every matrix/.append style={name=diag},
            nodes in empty cells,
            execute at end picture={
                \filldraw [myred, opacity=0.3] (diag-1-19) circle (15pt);
                \filldraw [myred, opacity=0.3] (diag-1-1) circle (15pt);
                \filldraw [myred, opacity=0.3] (diag-2-4) circle (15pt);
                \filldraw [myred, opacity=0.3] (diag-4-8) circle (15pt);
                \filldraw [myred, opacity=0.3] (diag-1-7) circle (15pt);
                \filldraw [myred, opacity=0.3] (diag-3-3) circle (15pt);
                \filldraw [myred, opacity=0.3] (diag-1-5) circle (15pt);
                \filldraw [myred, opacity=0.3] (diag-3-5) circle (15pt);
                \filldraw [myred, opacity=0.3] (diag-4-4) circle (15pt);
                \filldraw [myred, opacity=0.3] (diag-2-6) circle (15pt);
                \filldraw [blue, opacity=0.3] (diag-4-6) circle (15pt);
                \filldraw [blue, opacity=0.3] (diag-3-9) circle (15pt);
                \filldraw [blue, opacity=0.3] (diag-1-13) circle (15pt);
                \filldraw [blue, opacity=0.3] (diag-4-12) circle (15pt);
                \filldraw [blue, opacity=0.3] (diag-2-8) circle (15pt);
                \filldraw [blue, opacity=0.3] (diag-4-10) circle (15pt);
                \filldraw [blue, opacity=0.3] (diag-2-10) circle (15pt);
                \filldraw [blue, opacity=0.3] (diag-1-9) circle (15pt);
                \filldraw [blue, opacity=0.3] (diag-3-11) circle (15pt);
            }]
        {(0,3)}\drar\&\&
        {(4,7)}\drar\&\&
        {(8,11)}\drar\&\&
        {(12,15)}\drar\&\&
        {(1,16)}\drar\&\&
        {(2,5)}\drar\&\&
        {(6,9)}\drar\&\&
        {(10,13)}\drar\&\&
        {(14,17)}\drar\&\&
        {(0,3)}\&\&\\
        \&{(0,7)}\drar\urar\&\&
        {(4,11)}\drar\urar\&\&
        {(8,15)}\drar\urar\&\&
        {(1,12)}\drar\urar\&\&
        {(5,16)}\drar\urar\&\&
        {(2,9)}\drar\urar\&\&
        {(6,13)}\drar\urar\&\&
        {(10,17)}\drar\urar\&\&
        {(3,14)}\drar\urar\&\&\\
        {(7,14)}\drar\urar\&\&
        {(0,11)}\drar\urar\&\&
        {(4,15)}\drar\urar\&\&
        {(1,8)}\drar\urar\&\&
        {(5,12)}\drar\urar\&\&
        {(9,16)}\drar\urar\&\&
        {(2,13)}\drar\urar\&\&
        {(6,17)}\drar\urar\&\&
        {(3,10)}\drar\urar\&\&
        {(7,14)}\&\&\\
        \&{(11,14)}\urar\&\&
        {(0,15)}\urar\&\&
        {(1,4)}\urar\&\&
        {(5,8)}\urar\&\&
        {(9,12)}\urar\&\&
        {(13,16)}\urar\&\&
        {(2,17)}\urar\&\&
        {(3,6)}\urar\&\&
        {(7,10)}\urar\&\&
    \end{tikzcd}
    }
    \caption{AR quiver for $C_{-3}(A_4)$. The red discs indicate $\cA$, and the blue discs indicate $\Sigma \cA$.}
    \label{arq:c3a4}
    \end{figure}
    
    \indent Using \cite[thm. 4.6]{jorgensen2022abelian} gives $\cA \cong \mod(A)$, where $A = kQ/I$ with
    \begin{equation*}
        Q: \quad\begin{tikzcd}
        & 4\dar &\\
        3\rar["\alpha"] &
        2\rar["\beta"]  &
        1
    \end{tikzcd}
    \qquad \qquad
    I = \langle\alpha \beta\rangle.
    \end{equation*}
    Calculating the AR quiver of $A$ results in the following:
    \begin{equation*}
        \adjustbox{scale=.8}
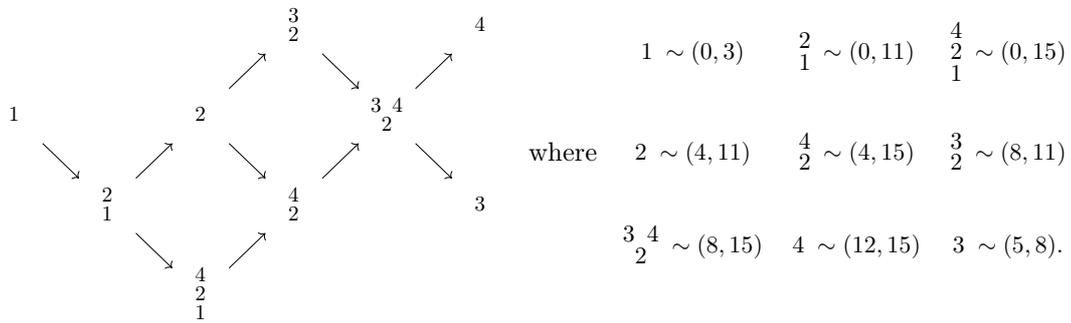
{%
    \begin{tikzcd}[row sep=5pt, column sep = 5pt]
        &&&\qrep{3\\2}\ar[dr,shorten=-3mm]&&\qrep{4}\\
        \qrep{1}\ar[dr,shorten=-3mm]&&\qrep{2}\ar[ur,shorten=-3mm]\ar[dr,shorten=-3mm]&&\qrep{3,,4\\,2,}\ar[ur,shorten=-3mm]\ar[dr,shorten=-3mm]\\
        &\qrep{2\\1}\ar[ur,shorten=-3mm]\ar[dr,shorten=-3mm]&&\qrep{4\\2}\ar[ur,shorten=-3mm]&&\qrep{3}\\
        &&\qrep{4\\2\\1}\ar[ur,shorten=-3mm]&&&
    \end{tikzcd}
    }
    \text{where}\hspace{3pt}
    \adjustbox{scale=.9}{%
        \(\begin{array}{ccc}
            \qrepbox{1}{1} \sim (0,3) &\qrepbox{1}{2\\1} \sim (0,11) &\qrepbox{1}{4\\2\\1} \sim (0,15)\\[30pt]
            \qrepbox{1}{2} \sim (4,11) &\qrepbox{1}{4\\2} \sim (4,15) &\qrepbox{1}{3\\2} \sim (8,11)  \\[23pt]

            \qrepbox{2}{3,,4\\,2,} \sim (8,15) &\qrepbox{1}{4} \sim (12,15) &\qrepbox{1}{3} \sim (5,8).\\
        \end{array}\)
    }
    \end{equation*}

    Using this we can calculate the torsion-free classes of $\cA$ (see \Cref{fig:d4_torf}).
    By \cref{cor:intermediate-bijection} these torsion-free classes will correspond exactly to the $\cA$-intermediate subcategories
    in $\pC_{-3}(A_4)$.
    As an example, choose $\cF = \add((0,3),(0,11), (4,11), (8,11))$, then $\pC = \Sigma\cF \ast \cA$ is an $\cA$-intermediate subcategory (see \Cref{arq:c3a4_intermediate1}).
    
    \begin{figure}[ht]
    \scalebox{.55}{
        \begin{tikzcd}[row sep=15pt, column sep = 0pt, ampersand replacement=\&,
            every matrix/.append style={name=diag},
            nodes in empty cells,
            execute at end picture={
                \filldraw [myred, opacity=0.3] (diag-1-19) circle (15pt);
                \filldraw [myred, opacity=0.3] (diag-1-1) circle (15pt);
                \filldraw [myred, opacity=0.3] (diag-2-4) circle (15pt);
                \filldraw [myred, opacity=0.3] (diag-4-8) circle (15pt);
                \filldraw [myred, opacity=0.3] (diag-1-7) circle (15pt);
                \filldraw [myred, opacity=0.3] (diag-3-3) circle (15pt);
                \filldraw [myred, opacity=0.3] (diag-1-5) circle (15pt);
                \filldraw [myred, opacity=0.3] (diag-3-5) circle (15pt);
                \filldraw [myred, opacity=0.3] (diag-4-4) circle (15pt);
                \filldraw [myred, opacity=0.3] (diag-2-6) circle (15pt);
                \filldraw [green, opacity=0.3] (diag-4-6) circle (15pt);
                \filldraw [green, opacity=0.3] (diag-3-9) circle (15pt);
                \filldraw [green, opacity=0.3] (diag-2-8) circle (15pt);
                \filldraw [green, opacity=0.3] (diag-4-10) circle (15pt);
                \filldraw [cyan, opacity=0.3] (diag-3-7) circle (15pt);
            }]
        {(0,3)}\drar\&\&
        {(4,7)}\drar\&\&
        {(8,11)}\drar\&\&
        {(12,15)}\drar\&\&
        {(1,16)}\drar\&\&
        {(2,5)}\drar\&\&
        {(6,9)}\drar\&\&
        {(10,13)}\drar\&\&
        {(14,17)}\drar\&\&
        {(0,3)}\&\&\\
        \&{(0,7)}\drar\urar\&\&
        {(4,11)}\drar\urar\&\&
        {(8,15)}\drar\urar\&\&
        {(1,12)}\drar\urar\&\&
        {(5,16)}\drar\urar\&\&
        {(2,9)}\drar\urar\&\&
        {(6,13)}\drar\urar\&\&
        {(10,17)}\drar\urar\&\&
        {(3,14)}\drar\urar\&\&\\
        {(7,14)}\drar\urar\&\&
        {(0,11)}\drar\urar\&\&
        {(4,15)}\drar\urar\&\&
        {(1,8)}\drar\urar\&\&
        {(5,12)}\drar\urar\&\&
        {(9,16)}\drar\urar\&\&
        {(2,13)}\drar\urar\&\&
        {(6,17)}\drar\urar\&\&
        {(3,10)}\drar\urar\&\&
        {(7,14)}\&\&\\
        \&{(11,14)}\urar\&\&
        {(0,15)}\urar\&\&
        {(1,4)}\urar\&\&
        {(5,8)}\urar\&\&
        {(9,12)}\urar\&\&
        {(13,16)}\urar\&\&
        {(2,17)}\urar\&\&
        {(3,6)}\urar\&\&
        {(7,10)}\urar\&\&
    \end{tikzcd}
    }
    \caption{An $\cA$-intermediate category of $C_{-3}(A_4)$. The red discs indicate $\cA$, the green discs indicate $\Sigma \cF$, and the cyan disc is not in either, but is contained in $\Sigma\cF \ast \cA$.}
    \label{arq:c3a4_intermediate1}
    \end{figure}
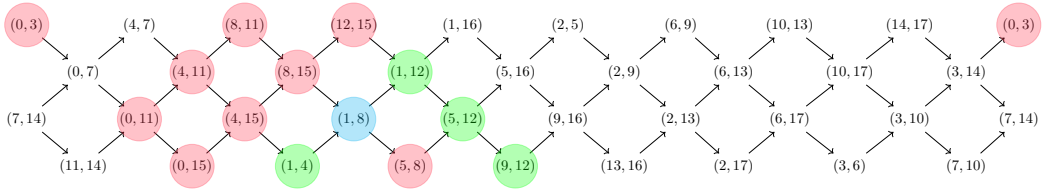

    \begin{figure}
    \begin{center}
    \begin{tikzpicture}[scale=.18, transform shape]
        \dFourTorsScope{(0,0)}{}{}{}{}{}{}{}{}{}

        \dFourTorsScope{(-15,-8)}{blue}{}{}{}{}{}{}{}{}
        \dFourTorsScope{(-5,-8)}{}{}{}{blue}{}{}{}{}{}
        \dFourTorsScope{(5,-8)}{}{}{}{}{}{}{}{blue}{}
        \dFourTorsScope{(15,-8)}{}{}{}{}{}{}{}{}{blue}

        \dFourTorsScope{(-25,-16)}{blue}{blue}{}{}{}{}{}{}{}
        \dFourTorsScope{(-15,-16)}{blue}{}{}{}{}{}{}{blue}{}
        \dFourTorsScope{(-5,-16)}{blue}{}{}{}{}{}{}{}{blue}
        \dFourTorsScope{(5,-16)}{}{}{}{blue}{}{blue}{}{}{}
        \dFourTorsScope{(15,-16)}{}{}{}{blue}{blue}{}{}{}{}
        \dFourTorsScope{(25,-16)}{}{}{}{}{}{}{}{blue}{blue}

        \dFourTorsScope{(30,-24)}{}{}{}{blue}{blue}{}{}{}{blue}
        \dFourTorsScope{(20,-24)}{}{}{}{blue}{}{blue}{}{blue}{}
        \dFourTorsScope{(10,-24)}{}{}{}{blue}{blue}{blue}{}{}{}
        \dFourTorsScope{(0,-24)}{blue}{}{}{}{}{}{}{blue}{blue}
        \dFourTorsScope{(-10,-24)}{blue}{blue}{}{}{}{}{}{blue}{}
        \dFourTorsScope{(-20,-24)}{blue}{blue}{blue}{}{}{}{}{}{}
        \dFourTorsScope{(-30,-24)}{blue}{blue}{}{blue}{}{}{}{}{}

        \dFourTorsScope{(-20,-32)}{blue}{blue}{blue}{blue}{}{}{}{}{}
        \dFourTorsScope{(-10,-32)}{blue}{blue}{blue}{}{}{}{}{blue}{}
        \dFourTorsScope{(0,-32)}{blue}{blue}{blue}{}{}{}{}{}{blue}
        \dFourTorsScope{(10,-32)}{blue}{blue}{}{blue}{}{blue}{}{}{}
        \dFourTorsScope{(20,-32)}{}{}{}{blue}{blue}{blue}{blue}{}{}

        \dFourTorsScope{(-25,-40)}{blue}{blue}{blue}{blue}{blue}{}{}{}{}
        \dFourTorsScope{(-15,-40)}{blue}{blue}{blue}{blue}{}{blue}{}{}{}
        \dFourTorsScope{(-5,-40)}{blue}{blue}{blue}{}{}{}{}{blue}{blue}
        \dFourTorsScope{(5,-40)}{blue}{blue}{}{blue}{}{blue}{}{blue}{}
        \dFourTorsScope{(15,-40)}{}{}{}{blue}{blue}{blue}{blue}{blue}{}
        \dFourTorsScope{(25,-40)}{}{}{}{blue}{blue}{blue}{blue}{}{blue}

        \dFourTorsScope{(-15,-48)}{blue}{blue}{blue}{blue}{blue}{blue}{}{}{}
        \dFourTorsScope{(-5,-48)}{blue}{blue}{blue}{blue}{blue}{}{}{}{blue}
        \dFourTorsScope{(5,-48)}{blue}{blue}{blue}{blue}{}{blue}{}{blue}{}
        \dFourTorsScope{(15,-48)}{}{}{}{blue}{blue}{blue}{blue}{blue}{blue}

        \dFourTorsScope{(0,-56)}{blue}{blue}{blue}{blue}{blue}{blue}{blue}{}{}

        \dFourTorsScope{(-5,-64)}{blue}{blue}{blue}{blue}{blue}{blue}{blue}{}{blue}
        \dFourTorsScope{(5,-64)}{blue}{blue}{blue}{blue}{blue}{blue}{blue}{blue}{}

        \dFourTorsScope{(0,-72)}{blue}{blue}{blue}{blue}{blue}{blue}{blue}{blue}{blue}
    \end{tikzpicture}
    \end{center}
    \caption{Torsion-free classes of $\cA$. Each small figure is in the shape of the AR-quiver for $\cA$ as described above, and the blue discs indicate the indecomposable objects of the torsion-free class.}
    \label{fig:d4_torf}
    \end{figure}
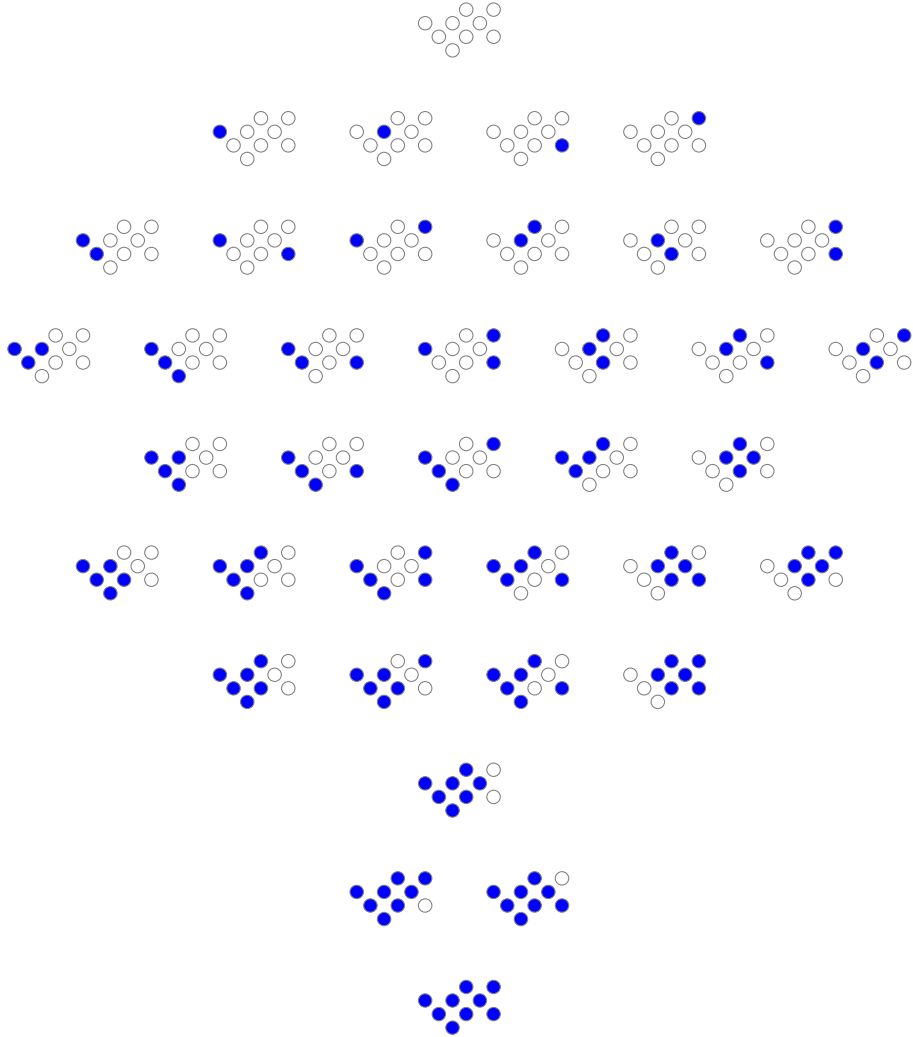

    \end{example}

    \begin{example}
        Consider the algebra $A\coloneqq kQ$, where

        \begin{equation*}
            Q:\quad
        \begin{tikzcd}
            1\rar&2 \rar&3.
        \end{tikzcd}
        \end{equation*}

        The AR-quiver of $D^b(A)$ is as follows

        \begin{equation*}
        \cdots
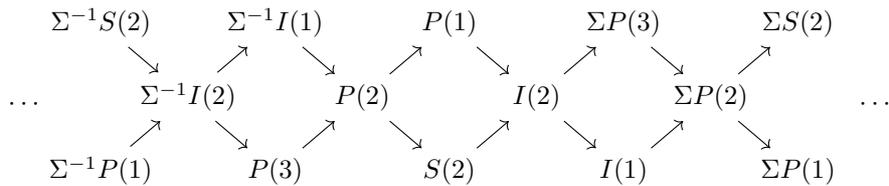
\hspace{1em}
        \begin{tikzcd}[column sep = 1em, row sep = 1em]
            \qrepp{\Sigma^{-1}S(2)}\ar[dr]&&\qrepp{\Sigma^{-1}I(1)}\ar[dr]&&\qrepp{P(1)}\ar[dr]&&\qrepp{\Sigma P(3)}\ar[dr]&&\qrepp{\Sigma S(2)} \\
            &\qrepp{\Sigma^{-1}I(2)}\ar[dr]\ar[ur]&&\qrepp{P(2)}\ar[ur]\ar[rd]&&\qrepp{I(2)}\ar[dr]\ar[ur]&&\qrepp{\Sigma P(2)}\ar[dr]\ar[ur]\\
            \qrepp{\Sigma^{-1}P(1)}\ar[ur]&&\qrepp{P(3)}\ar[ur]&&\qrepp{S(2)}\ar[ur]&&\qrepp{I(1)}\ar[ur]&&\qrepp{\Sigma P(1)}\\
        \end{tikzcd}
        \hspace{1em}\cdots
        \end{equation*}
         
        Note that $\mod A$ sits inside $D^b(A)$ as a proper abelian subcategory 
        since $\mod A$ is the heart of the standard t-structure in $D^b(A)$,
        but one can also find other proper abelian subcategories in $D^b(A)$ which are not hearts
        of t-structures.

        \indent For example, $\cS = \set{P(3), S(2)}$ is a 2-orthogonal collection (see \cite[def. 1.2]{jorgensen2022abelian} for a definition).
        Thus the extension closure $\langle\cS\rangle = \add(P(3), S(2), P(2))$ is a proper abelian
        subcategory by \cite[thm. A]{jorgensen2022abelian}.
        Notice that $\langle\cS\rangle\cong \mod kA_2$.
        It is easy to check that $\langle\cS\rangle$ satisfies $E_2$.
        This means that the proper abelian subcategory $\langle\cS\rangle$ satisfies \Cref{setup}.
        Furthermore,
        \begin{equation*}
           \Sigma \langle\cS\rangle \ast \langle\cS\rangle
           = \Sigma \langle\cS\rangle \oplus \langle\cS\rangle 
           =  \langle\cS\rangle \ast \Sigma\langle\cS\rangle.
        \end{equation*}

        Since we know that $\langle\cS\rangle\cong \mod kA_2$ we can find torsion-free classes. 
        As an example we can choose $\cF = \add(P(3),P(2))$, which would give the corresponding $\langle\cS\rangle$-intermediate category $\pC = \Sigma \cF \ast \langle\cS\rangle = \add(\Sigma P(3), \Sigma P(2), P(3), P(2), S(2))$.
  \end{example}

    \textit{Acknowledgement.}
    Thanks to my supervisor Peter Jørgensen for all his guidance.

    \indent This project was supported by grant no. DNRF156 from the Danish National Research Foundation.

\end{document}